\documentclass[12pt]{article}
\usepackage{amsmath,amssymb}
\usepackage[english]{babel}
\usepackage[utf8]{inputenc}
\usepackage{epsfig}
\usepackage{graphicx}
\usepackage[all]{xy}
\usepackage{lmodern}
\usepackage{stackengine}
\usepackage{latexsym, amsmath, amssymb, amsthm, mathrsfs, framed}
\usepackage{paralist}
\usepackage[colorlinks=true,linkcolor=blue,urlcolor=blue, citecolor=blue]%
{hyperref}
\usepackage{bbm}
\usepackage{mathptmx}
\usepackage{microtype}
\usepackage{enumitem}
\usepackage{mathtools}
\usepackage{comment}

\DeclareMathAlphabet{\mathcal}{OMS}{cmsy}{m}{n}

\newtheorem{teo}{Theorem}[section]
\newtheorem*{teo*}{Theorem}
\newtheorem{cor}[teo]{Corolary}
\newtheorem{prop}[teo]{Proposition}

\newtheorem{lemma}[teo]{Lemma}

\theoremstyle{remark}
\newtheorem{rmk}[teo]{Remark}

\newcommand{\Id}{{\rm Id}}

\newcommand{\tr}{\mathop{{\rm tr}}}
\newcommand{\adj}{\mathop{\rm adj}}

\newcommand{\vol}{\mathop{\rm vol}}

\renewcommand{\adj}{{\rm adj}\,}

\newcommand{\C}{{\mathbb C}}

\newcommand{\R}{{\mathbb R}}

\newcommand{\E}{{\mathbb E}}

\renewcommand{\S}{{\mathbb S}}

\renewcommand{\imath}{{\bf i}}

\newcommand{\tor}{\xymatrix{\ar@{-->}[r]&}}
\newcommand{\mapstor}{\xymatrix{\ar@{|-->}[r]&}}

\DeclareMathOperator{\Ker}{Ker}

\title{On the logarithmic energy of solutions to the polynomial eigenvalue problem}
\author{Diego Armentano\thanks{Facultad de Ciencias Económicas y de Administración, Universidad de la República, Uruguay. E-mails: \{diego.armentano, federico.carrasco\}@fcea.edu.uy.} \and Federico Carrasco\footnotemark[1]
\and Marcelo Fiori\thanks{Facultad de Ingeniería, Universidad de la República, Uruguay. E-mail: mfiori@fing.edu.uy.}}
\date{}

\begin{document}

\maketitle

\begin{abstract}
   In this paper, we compute the expected logarithmic energy of solutions to the polynomial eigenvalue problem for random matrices. We generalize some known results for the Shub-Smale polynomials, and the spherical ensemble. These two processes are the two extremal particular cases of the polynomial eigenvalue problem, and we prove that the logarithmic energy lies between these two cases. In particular, the roots of the Shub-Smale polynomials are the ones with the lowest logarithmic energy of the family.
\end{abstract}



\section{Introduction and Main Result}

The problem of finding configurations of points in the 2-dimensional sphere with {small logarithmic energy} is a very challenging problem, with several applications. It is one of the problems listed by Smale for the XXI Century \cite{smale_problems}, and there have been several advances in different directions related to this problem.

Given $N$ points in $\R^3$, the logarithmic energy of the configuration is defined as
$$V(x_1,\ldots,x_N) = -\sum_{1\leq i < j \leq N} \ln \|x_i-x_j\| .$$
The problem of minimizing this energy in the unit sphere $\S^2$ is considered a very hard optimization problem, also known as the Fekete problem. Not only are the configurations of points that minimize the energy not completely understood even for a small number of points (for instance, $N=7$), but also the asymptotic value of the minimum is not known with enough precision. More precisely, let
$$V_N = \min_{x_1,\ldots,x_N \in \S^2} V(x_1,\ldots,x_N)$$
be the minimum of the energy in the sphere. The 7th Smale problem consists of finding a configuration of points $x_1,\ldots,x_N$ in the sphere, in polynomial time in $N$, such that its logarithmic energy $V(x_1,\ldots,x_N)$ is close enough to the minimum, namely $V(x_1,\ldots,x_N) - V_N \leq c\ln N$, for a universal constant $c$. 

One of the major obstacles is that the value of $V_N$ itself is not known with precision up to the $\ln N$ term, and therefore problem number 7 of Smale's list is still far from being solved.


Indeed, the value of $V_N$ is \cite{betermin2018renormalized}
\begin{equation}\label{minimo_VN}
V_N = \frac{\kappa}{2} N^2 - \frac{N\ln N}{4} + CN + o(N) ,
\end{equation}
where $\kappa = \frac{1}{2}-\ln 2$ and  $C$ is an unknown constant. As far as it is known, this constant $C$ is bounded (lower bound by \cite{Lauritsen,carlos_fatima}, upper bound by \cite{brauchart2012,betermin2018renormalized})
$$-0.028426\ldots = \frac{\ln(2)}{2}-\frac{3}{8}\leq C \leq \ln 2 + \frac{1}{4}\ln \frac{2}{3} + \frac{3}{2}\ln \frac{\sqrt{\pi}}{\Gamma(1/3)} = -0.027802\ldots,$$
and the upper bound is conjectured to be actually the value for $C$ \cite{brauchart2012,betermin2018renormalized}.

Despite the intrinsic difficulties of finding these optimal configurations of points, or even the value of the minimal energy, there have been some very exciting advances throughout the last decades. For instance, the diamond ensemble proposed by Beltrán and Etayo \cite{beltran2018diamond}, which achieves configurations of points with logarithmic energy very close to the conjectured minimum, and two other random processes, which we describe in more detail in what follows.

The first one, proposed by Armentano, Beltrán, and Shub in \cite{ABS}, consists of taking the roots of a random polynomial. Specifically, let \begin{equation}\label{shub_smale_poly}
p_N(z) = \sum_{k=0}^N a_k \binom{N}{k}^{1/2}z^k
\end{equation}
with $a_k$ i.i.d. complex standard Gaussian coefficients $\mathcal{N}_{\C}(0,1)$, i.e., each coefficient is $\alpha + i\beta$, with $\alpha,\beta$ real independent zero-mean and $1/2$ variance Gaussian. Now compute the roots of $p_N$ in $\C$, and project them to $\S^2$ through the stereographic projection.
The authors prove that the expected logarithmic energy of the resulting ensemble in $\S^2$ is
\begin{equation}\label{energia_abs}
\frac{\kappa}{2} N^2 - \frac{N\ln N}{4} - \frac{\kappa}{2} N .
\end{equation}
 Observe that the expression coincides up to the first two terms with \eqref{minimo_VN}, and the constant for the linear term is $-\frac{\kappa}{2}\approx 0.096\ldots$.\\
 More recently, in \cite{Fluctuations} the authors prove a central limit theorem for the logarithmic energy resulting from this random process, {where they show that the fluctuations are of order $\sqrt{N}$, and therefore a typical realization of this process will have energy close to the expression in \eqref{energia_abs}}.

The second approach consists of taking the eigenvalues of a random matrix. Specifically, let $A$ and $B$ be two random matrices with i.i.d. complex standard Gaussian entries $\mathcal{N}_{\C}(0,1)$. Now compute the eigenvalues of the matrix $AB^{-1}$, and project them to $\S^2$ through the stereographic projection. Alishahi and Zamani \cite{alishahi2015spherical} proved that the expected value of the logarithmic energy for this configuration, so-called \emph{spherical ensemble}, is
\begin{equation}\label{energia_AZ}
\frac{\kappa}{2} N^2 - \frac{N\ln N}{4} + \left(\frac{\ln 2}{2} -\frac{\gamma}{4} \right)N -\frac{1}{8} + O\left(\frac{1}{N} \right),
\end{equation}
where $\gamma = 0.57721\ldots$ is the Euler-Mascheroni constant. Observe that the first two terms coincide with the known expression in \eqref{minimo_VN}, and the constant for the linear term is $\left(\frac{\ln 2}{2} -\frac{\gamma}{4} \right) \approx 0.2022\ldots$.

In this paper, we study a strategy for producing configurations of points, for which the last two examples (zeros of random polynomials and eigenvalues of random matrices) are particular cases. Namely, let us consider the following polynomial in $\C$,
$$
F(z)=\det\left(\sum_{i=0}^d G_i{d\choose i}^{1/2}z^i\right),
$$
where each $G_i$ is an $r\times r$ random matrix with independent entries distributed as $\mathcal{N}_\C(0,1)$.
The problem of finding the zeros of this function is known as the \textit{polynomial eigenvalue problem} (PEVP). Observe that $F(z)$ has, generically, $N=dr$ roots in $\C$, which can be projected to the unit sphere through the stereographic projection as before. We will denote the PEVP-ensemble to this configuration of points in $\S^2$.

Now, for a given number of points $N$, one can choose different pairs of its divisors $(d,r)$ forming $N=dr$. Notably, for $r=1$ and $d=N$, we obtain exactly the random polynomials as in \eqref{shub_smale_poly}. In the other extreme, for $r=N$ and $d=1$, we obtain $F(z) = \det(G_0+G_1z)$, whose roots coincide with the eigenvalues of $-G_0G_1^{-1}$, and therefore we recover the spherical ensemble.

Some numerical experiments suggest that the expected logarithmic energy of intermediate instances (meaning $1<d<N$) lies between the energy of the two extremal cases and decreases linearly with $d$. The main result of this paper, which we state below, gives a precise computation of the expected logarithmic energy for the PEVP-ensemble. The numerical experiments, along with the analysis of this dependence on $d$, are presented in Section \ref{sec:discussion}.

\begin{teo}[Main Theorem]\label{main} Let $F(z)$ be the random complex polynomial of degree $N$ defined as 
$$
F(z)=\det\left(\sum_{i=0}^d G_i{d\choose i}^{1/2}z^i\right),
$$
where $G_i$ are $r \times r$ matrices with i.i.d. entries following $\mathcal{N}_\C(0,1)$. Then, with the definitions above, we have
$$
\E(V(x_1,\cdots,x_N))=\frac{\kappa}{2}N^2-\frac{N\log d}{4}-\frac{N}{4}\bigg(1+\psi(r+1)-\psi(2) -2\ln 2\bigg)
$$
where $\psi(n)=\frac{\Gamma'(n)}{\Gamma(n)}$ is the digamma function, i.e., the logarithmic derivative of the gamma function $\Gamma(n)$.
\end{teo}

(See Section \ref{proofmain} for the proof.)

\medskip

Observe that this result generalizes the computed expected logarithmic energies of the ensembles by \cite{alishahi2015spherical} and \cite{ABS}.
Moreover, for $r=N, d=1$, we get


\begin{eqnarray*}
\E(V(x_1,\cdots,x_N)) &=&\frac{\kappa}{2}N^2-\frac{N\psi(N+1)}{4}+N\left(\frac{\ln 2}{2} -\frac{\gamma}{4}\right).
\end{eqnarray*}

This is actually the exact value for the expected value of the spherical ensemble, which, to the best of our knowledge, had not been computed before. Using the usual approximation of $\psi(N+1)$, we obtain the same asymptotic expression as in \eqref{energia_AZ}. A more detailed asymptotic analysis of this expression is given in Section \ref{sec:discussion}.


\begin{rmk}\label{invariant_ortho}
Observe that in the matrix $\sum_{i=0}^d G_i{d\choose i}^{1/2}z^i$, each entry is a Shub-Smale polynomial, like the ones in \eqref{shub_smale_poly}. The distribution of the roots of these polynomials, projected to $\S^2$, are invariant under the orthogonal group. Now, since the determinant is a homogeneous polynomial in the entries of the matrix, the zeros of the resulting process $F(z)$ are also invariant under the same group of isometries (Proposition 2.1.1 of \cite{krishnaphd}). This invariance in $\S^2$ is a desirable property if we want the resulting points of a process in which every configuration is possible to be well distributed in the sphere.
\end{rmk}

The rest of the paper is organized as follows. The proof of the main theorem reduces to compute the expectation of the three terms in \eqref{tres_terminos}, which present uneven levels of difficulties. Indeed, the results of the first and last terms follow from relatively simple random matrix computations, while the second one requires more tools.

In order to ease the notation, and also to state more general results that can be useful in other contexts, we first state and prove some tools in the framework of random matrices in Section \ref{sec:techincal_tools}. Then, in Section \ref{computations}, we instantiate some of these results to obtain the expectation of the two most straightforward terms and use the weighted Kac-Rice formula, along with other tools, to compute the remaining expectation. We aggregate these computations in Section \ref{proofmain} to complete the proof of the main theorem. We conclude in Section \ref{sec:discussion} with some experimental results and a discussion on the dependence on $d$ of the expected logarithmic energy.

\section{Technical tools and proof of Main Theorem}

In what follows, we compute the expected logarithmic energy of the PEVP-ensemble, but instead of doing so in the unit sphere, we will derive the expression through the Riemann sphere (i.e., the sphere of radius $1/2$ centered at $(0,0,1/2)$ as in \cite{B3}). Throughout the paper, we will denote as $z_1,\ldots,z_N$ the roots of $F(z)$ in $\C$, and as $\hat z_1,\ldots, \hat z_N$ their stereographic projections onto the Riemann sphere. We can then transform this configuration to the unit sphere $\S^2$ via $$(a,b,c) \to (2a,2b,2c-1).$$

Given a configuration of points $\hat z_1,\ldots,\hat z_N$ in the Riemann sphere, it is easy to see that the logarithmic energies of this configuration and the corresponding configuration in $\S^2$ are related as
\begin{equation}\label{exps2}
V(x_1,\cdots,x_N) = V(\hat z_1,\cdots,\hat z_N) - \frac{N(N-1)}{2}\ln 2.
\end{equation}

The goal is then to compute $\E(V(\hat z_1,\cdots,\hat z_N))$. Following Armentano et al. \cite[Proposition 1]{ABS}, since $F(z)$ is a polynomial, the logarithmic energy can be decomposed as follows:
$$
V(\hat z_1,\cdots,\hat z_N)=(N-1)\sum_{i=1}^{N}\log\sqrt{1+|z_i|^2}-\frac{1}{2}\sum_{i=1}^{N}\log|F'(z_i)|+\frac{N}{2}\log|a_N|,
$$
where $a_N$ is the leading coefficient of $F(z)$. It is easy to see that this leading coefficient is $\det(G_d)$.

Then, taking expectation we get that $\mathbb E(V(\hat z_1,\cdots,\hat z_N))$ is equal to
\begin{equation}\label{tres_terminos}
(N-1)\mathbb E\left(\sum_{i=1}^{N}\log\sqrt{1+|z_i|^2}\right)-\frac{1}{2}\mathbb E\left(\sum_{i=1}^{N}\log|F'(z_i)|\right)+\frac{N}{2}\mathbb E(\log|\det(G_d)|).  
\end{equation}
By computing the expectation of the three terms in the previous expression, we obtain the main result of the paper.

In order to do so, we will first prove some useful results in the following part, and then compute the mentioned expectations in Section \ref{computations}.

\subsection{Technical tools}
\label{sec:techincal_tools}
This section is dedicated to presenting (and proving) the technical tools (and results) necessary for the computations. For those eager to understand the underlying concepts without delving into the technical details, we refer them to Section \ref{computations}.


\subsubsection{Geometric and probabilistic tools}

Firstly, let us recall the coarea formula and derive some known results for random variables.

\medskip

Suppose $\varphi:M\to N$ is a surjective map from a Riemannian manifold $M$ to a Riemannian manifold $N$, whose derivative $D\varphi(x):T_xM\to T_{\varphi(x)}N$ is surjective for almost all $x\in M$.
The horizontal space $H_x\subset T_xM$ is defined as the orthogonal complement of $\Ker D\varphi(x)$.
The horizontal derivative of $\varphi$ at $x$ is the restriction of $D\varphi(x)$ to $H_x$.
The \textit{Normal Jacobian} $NJ_{\varphi}(x)$ is the absolute value of the determinant of the horizontal derivative, defined almost everywhere on $M$.

\begin{teo*}[\cite{BCSS} {p. 241} The coarea formula]\label{coarea}
  Let $M,N$ be Riemannian manifolds of respective dimension $m\geq n$, and let $\varphi:M\to N$ be a smooth surjective map, whose derivative $D\varphi(x):T_xM\to T_{\varphi(x)}N$ is surjective for almost all $x\in M$.
  Then, for any positive measurable function, $\Phi:M\to[0,+\infty)$ we have
  \[
  \int_{x\in M}\Phi(x)dM=\int_{y\in N}\int_{x\in\varphi^{-1}(y)}\frac{\Phi(x)}{NJ_\varphi(x)}d\varphi^{-1}(y)dN
  \]
  and
  \[
  \int_{x\in M}NJ_\varphi(x)\Phi(x)dM=\int_{y\in N}\int_{x\in\varphi^{-1}(y)}\Phi(x)d\varphi^{-1}(y)dN,
  \]
  where $d\varphi^{-1}(y)$ is the induced volume measure in the manifold $\varphi^{-1}(y)$.
\end{teo*}

\begin{lemma}
    Let $X(\cdot):\C^D \to\C^d$ be a random field with density $\rho_X(x)$, and $\varphi:\C^d\to \C^l$ be a smooth surjective map, whose derivative $D\varphi(x):\C^d\to \C^l$ is surjective for almost all $x\in \C^d$. 
    Then, the {pointwise density} of the random field $Y(\cdot):\C^D\to\C^l$ defined by $Y(\cdot)=\varphi(X(\cdot))$ is
    \[
    \rho_Y(y)=\int_{\varphi^{-1}(y)}\rho_X(x)\frac{1}{NJ_{\varphi}(x)}d\varphi^{-1}(y).
    \]
\end{lemma}

\begin{proof}
    It follows from the coarea formula, applied to the fact that
    \[
    \mathbb P(Y\in B)=\mathbb P\left(X\in \varphi^{-1}(B)\right)=\int_{\varphi^{-1}(B)}\rho_X(x)dx,
    \] for every $B\subset \R^l$ Borel subset.
\end{proof}

As a direct application, we get the following result.

\begin{cor}\label{densF}
Let $A\in\C^{r\times r}$ be a random matrix with i.i.d. centered Gaussian entries $\mathcal N_{\C}(0,\sigma^2)$.
Then, the density at $0$ of the random variable $\det(A)$ is equal to
$$
\rho_{\det(A)}(0)=\frac{1}{\pi \sigma^{2r}\Gamma(r)}. 
$$
\end{cor}

\begin{proof}
    Using the previous lemma, it follows that
    $$
    \rho_{\det(A)}(0)=\frac{1}{\pi^{r^2}\sigma^{2r^2}}\int_{\Sigma}e^{\frac{-\|A\|^2}{\sigma^2}}\frac{1}{NJ_{\det}(A)}d\Sigma
    $$
    where $\Sigma=\{A\in\C^{r\times r} \ : \ \det(A)=0\}$.

    Following the Jacobi's formula, the derivative of the determinant at $A$ in the direction $\dot A$ is equal to $\tr(\adj(A)\dot A)$, where the adjugate matrix of $A$, $\adj(A)$, is the transpose of the cofactor matrix.
    
    After a routine computation, it follows that 
    \[
    NJ_{\det}(A)=\|\adj(A)\|_F^2,
    \]
    where $\|\cdot\|_F$ stands for the
    Frobenius norm.

    Let $\hat\Sigma$ be the set of rank $r-1$ matrices of $\C^{r\times r}$. It is clear that $\hat\Sigma$ is the smooth part of the algebraic variety $\Sigma$. 
    It follows that integrating over $\Sigma$ is the same as integrating over $\hat\Sigma$, then
    \begin{equation}\label{eq:rho0}
    \rho_{\det(A)}(0)=\frac{1}{\pi^{r^2}\sigma^{2r^2}}\int_{\hat\Sigma}e^{\frac{-\|A\|^2}{\sigma^2}}\frac{1}{\|\adj(A)\|_F^2}d\hat\Sigma.
    \end{equation}

    Given a matrix $A \in \hat{\Sigma}$, the kernel is a one dimensional subspace of $\C^r$ given by the orthogonal complement of the column space of $A$. Furthermore, given a direction $v\in\mathbb P(\C^r)$, it is easy to check that there exist a matrix $A\in\hat\Sigma$ such that the column space is orthogonal to $v$. Therefore, there is a well defined smooth surjective submersion
    \begin{equation}\label{def_phi}
    \varphi:\hat\Sigma\to\mathbb P(\C^r), \qquad \varphi(A)=v, \qquad \text{such that \ } Av=0,  
    \end{equation}
    {where $\mathbb P(\C^r)$ is the projective space, i.e., the space of one dimensional subspaces of $\C^r$.}
    
    A straightforward computation, using the derivative of an implicit function, shows that the differential map $D\varphi(A):T_A\hat\Sigma\to T_v \mathbb {P}(\C^r)$ is defined by 
    \[
    D\varphi(A)\dot A=-A^\dagger \dot A v,
    \]
    {where $A^\dagger$ is the Moore-Penrose pseudo-inverse of $A$.}
    The tangent space to $\hat\Sigma$ at $A$, can be parametrized as $\{\dot X A+A\dot Y \ : \dot X,\dot Y\in\C^{r\times r}\}$ (see Arnold et al. \cite{arnold}).

    It is easy to check that the hermitian complement to $\ker D\varphi(A)$ is the image of $\{(0,\dot wv^*) \ : \ \dot w^*v=0\}$ by the map $(\dot X,\dot Y)\mapsto \dot X A+A\dot Y$. 
    So, applying $D\varphi(A)$, we get the folowing composition 
    \[
    \{(0,\dot wv^*) \ : \ \dot w^*v=0\}\mapsto A\dot wv^*\mapsto \dot w,
    \]
    which is the identity.
    
    Then, the normal jacobian of the map $\varphi$ is the inverse of the normal jacobian of the first map, which is in fact the square of the product of the non-zero singular values of $A$.
    In conclusion, the normal jacobian of $\varphi$ satisfies,
    \[
    NJ_{\varphi}(A)=\frac{1}{\|\adj(A)\|^2_F}.
    \]
    {Applying the coarea formula to $\varphi:\hat{\Sigma}\to \mathbb{P}(\C^r)$, equation \eqref{eq:rho0} can be rewritten as}
    \begin{align*}
        \rho_{\det(A)}(0)=&\frac{1}{\pi^{r^2} \sigma^{2r^2}}\int_{v \in \mathbb P(\C^r)}\int_{A \in \C^{r\times(r-1)}}e^{\frac{-\|A\|_F^2}{\sigma^2}}dA d\mathbb P(\C^r).
    \end{align*}  
{Recall that the volume of $\mathbb P(\C^r)$ is $\frac{\pi^{r-1}}{\Gamma(r)}$. 
Noting that the inner intergal does not depend on $v$, it follows that
\begin{align*}
        \rho_{\det(A)}(0)=&\frac{\pi^{r-1}}{\pi^{r^2} \sigma^{2r^2}\Gamma(r)}\int_{A \in \C^{r\times(r-1)}}e^{\frac{-\|A\|_F^2}{\sigma^2}}dA.
    \end{align*}  

Finally, applying Fubini's theorem and taking polar coordinates, we get
\[\rho_{\det(A)}(0)   =\frac{1}{\pi \sigma^{2r}\Gamma(r)}.\]
}

\end{proof}

\begin{rmk}
    Observe that in this case, where the matrix A has Gaussian entries, this can also be proven using Kostlan's theorem, as in \cite[Thm. 4.7.3]{krishna}.
\end{rmk}

\begin{lemma}\label{eqcondexps}
Let $A\in\C^{r\times r}$ be a random matrix with i.i.d. centered Gaussian entries $\mathcal N_{\C}(0,\sigma^2)$.
If $\phi:\C^{r\times r}\to[0,+\infty)$ is a measurable function such that $\phi(MU)=\phi(M)$ for any unitary matrix $U$, then
$$
\underset{A\in\C^{r\times r}}{\mathbb E}\left(\phi(A) \big| \det(A)=0\right)=\underset{A\in\C^{r\times r}}{\mathbb E}\left(\phi(A)\big| A_r=0\right)
$$
where $A_r$ is the $r$-th column of $A$.
\end{lemma}

\begin{proof}
    By definition, the conditional expectation of $\phi(A)$ conditional to $\{\det(A)=0\}$ is equal to the integral of $\phi(A)$ with respect to the conditional density, i.e., 
    $$
    \mathbb E\left(\phi(A)  \big| \det(A)=0\right)=\frac{1}{\rho_{\det(A)}(0)}\int_\Sigma \frac{\phi(A)}{NJ_{\det}(A)}\frac{e^{\frac{-\|A\|^2}{\sigma^2}}}{\pi^{r^2}\sigma^{2r^2}}d\Sigma.
    $$
    Since $\rho_{\det(A)}(0)=\frac{1}{\pi \sigma^{2r}\Gamma(r)}$, $\phi(AU)=\phi(A)$ for every unitary matrix $U$, {and applying the coarea formula to the same map $\varphi$ in \eqref{def_phi} as in the previous corollary}, we get 
    $$
    \mathbb E\left(\phi(A)  \big| \det(A)=0\right)=\int_{\C^{r\times(r-1)}} \phi(A)\frac{e^{\frac{-\|A\|^2}{\sigma^2}}}{\pi^{r(r-1)}\sigma^{2r(r-1)}}dA=\mathbb E\left(\phi(A)  \big| A_r=0\right).
    $$
\end{proof}

\subsubsection{Random matrices computations}

The proofs of the following lemmas follow the ideas of Azaïs-Wschebor \cite{AW}, where $|\det A|$ is considered as the volume of the parallelotope in $\C^r$ generated by the columns $A_1,\cdots, A_r$ of $A$.

Observe that $\vol(A_1,\cdots,A_k)=\vol(A_1,\cdots,A_{k-1})\|\Pi_{V_{k-1}^\perp}A_k\|$, where $V_{k-1}$ is the subspace generated by $A_1,\cdots, A_{k-1}$, and $\|\Pi_{V_{k-1}^\perp}A_k\|$ is the Euclidean norm in $\C^r$ of the orthogonal projection of the vector $A_k$ onto the hermitian complement subspace $V_{k-1}$.

\begin{lemma}\label{explogdet}
    Let $A$ be a $r\times r$ random complex matrix with standard Gaussian i.i.d. entries.
    Then
    \[
    \mathbb E(\log|\det A|)=\frac{r(\psi(r+1)-1)}{2},
    \]
    where \(\psi(n)=\frac{\Gamma'(n)}{\Gamma(n)}\) is the digamma function.
\end{lemma}

\begin{proof}
    Using the linearity of the expectation and the properties of the logarithm, iterating this process, we conclude that 
    \begin{equation}\label{voltrick}
    \mathbb E(\log(\vol(A_1,\cdots,A_r)))=\displaystyle\sum_{i=1}^{r}\underset{v_i\in \C^i}{\mathbb E}(\log(\|v_i\|)),
    \end{equation} 
    where $v_i$ is a standard Gaussian complex vector in $\C^i$, and abusing notation, $\|v_i\|$ stands for the Euclidean norm of $v_i$ in its corresponding space.

    Using polar coordinates,
    \[
    \underset{v_i\in \C^i}{\mathbb E}(\log(\|v_i\|))=\frac{2}{\Gamma(i)}\int_0^\infty \rho^{2i-1}e^{-\rho^2} \log\rho d\rho=\frac{\psi(i)}{2}, 
    \]
{where the last equality follows from taking \(t=\rho^2\) and differentiating the integral expression of \(\Gamma(i)\) with respect to \(i\), or see directly Gradshteyn and Ryzhik \cite[4.352.4]{Gradshteyn}.}

    Using that $\psi(n) = \left(\displaystyle\sum_{i=1}^{n-1}\frac{1}{i}\right)-\gamma$, and that $n\psi(n+1) = n\psi(n)+1$, it follows

    \[
    \sum_{i=1}^r \psi(i) = r(\psi(r+1)-1).  
    \]
    Then, we conclude
    \begin{align*}
        \mathbb E(\log(|\det A|))&=\frac{r(\psi(r+1)-1)}{2}.
    \end{align*}
\end{proof}


\begin{lemma}
    Let $A$ be a $r\times r$ random complex matrix with standard Gaussian i.i.d. entries.

    Then, $\underset{A\in \C^{r\times r}}{\mathbb E}\left(|\det A|^2\log|\det A|\right)$ is equal to 
    $$
    \frac{r!}{2}\bigg((r+1)\psi(r+2)-r-\psi(2)\bigg).
    $$ 
\end{lemma}

\begin{proof}
    Let us define, 
    \[
    I_r(k)\coloneqq \underset{A_i\in \C^r}{\mathbb E}\left(\vol(A_1,\cdots,A_k)^2\log\left(\vol(A_1,\cdots,A_k)\right)\right)
    \]
    \[
    V_r(k)\coloneqq \underset{A_i\in \C^r}{\mathbb E}\left(\vol(A_1,\cdots,A_k)^2\right)
    \]
    where the $A_i$'s are independent standard Gaussian random vectors in $\C^r$.

    {Observe that 
    \begin{equation}\label{eq:recurrencia}
        \vol(A_1,\cdots,A_k)=\vol(A_1,\cdots,A_{k-1})\|\Pi_{V_{k-1}^\perp}A_k\|,
    \end{equation}
     where $V_{k-1}$ is the subspace generated by $A_1,\cdots, A_{k-1}$.
     Now, using the independence of the random vector, and proceeding iteratively, it is easy to check that
    }
    \begin{align*}\label{exp_iter}
    I_r(k)=&I_r(k-1)\underset{v\in \C^{r-k+1}}{\mathbb E}\|v\|^2+V_r(k-1)\underset{v\in \C^{r-k+1}}{\mathbb E}\left(\|v\|^2\log\|v\|\right)\\
    \vdots&\\
    =&I_r(1)\prod_{i=r-k+1}^r\underset{v\in \C^{i}}{\mathbb E}\|v\|^2+\frac{\Gamma(r+1)}{\Gamma(r-k+1)}\sum_{i=r-k+1}^{r-1}\frac{1}{i}\underset{v\in \C^i}{\mathbb E}\left(\|v\|^2\log\|v\|\right)\\
    =&\frac{\Gamma(r+1)}{\Gamma(r-k+1)}\sum_{i=r-k+1}^r\frac{1}{i}\underset{v\in \C^i}{\mathbb E}\left(\|v\|^2\log\|v\|\right).
    \end{align*}

    Using the iterative formula for $I_r(r)$, we get that the expectation of $|\det A|^2\log|\det A|$ is equal to
    \begin{align*}
    \underset{A\in \C^{r\times r}}{\mathbb E}\left(|\det A|^2\log|\det A|\right)=\frac{r!}{2}\left((r+1)\psi(r+2)-r-\psi(2)\right).
    \end{align*}
\end{proof}

\begin{cor}\label{corexpdet}
    Let $A$ be a $r\times r$ random complex matrix with standard Gaussian i.i.d. entries and $M$ be a fixed $r \times r$ matrix.

    Then, $\underset{A\in \C^{r\times r}}{\mathbb E}\left(|\det AM|^2\log|\det AM|\right)$ is equal to 
    $$
    |\det M|^2\frac{r!}{2}\bigg((r+1)\psi(r+2)-r-\psi(2)+\log|\det M|^2\bigg).
    $$ 
\end{cor}

\begin{proof}
    Using the multiplicative nature of the determinant, the properties of the logarithm, and the linearity of the expectation, one gets that $\underset{A\in \C^{r\times r}}{\mathbb E}\left(|\det AM|^2\log|\det AM|\right)$ is equal to
    \[
    |\det M|^2\left(\underset{A\in \C^{r\times r}}{\mathbb E}\left(|\det A|^2\log|\det A|\right)+\log|\det M|\underset{A\in \C^{r\times r}}{\mathbb E}|\det A|^2\right).
    \]
    {Now observe that using the same recursive argument as in \eqref{eq:recurrencia}, one gets that $\displaystyle{\underset{A\in \C^{r\times r}}{\mathbb E}|\det A|^2 = \prod_{i=1}^r \mathbb E_{v_i \in \C^i} \|v_i\|^2 } = r!$, where again $v_i$ is a standard gaussian vector in $\C^i$. The result follows. }
\end{proof}

\subsection{Three little terms}\label{computations}

We are now in conditions to compute the three terms of \eqref{tres_terminos}.

\medskip

Let us focus now on the first and last terms, which follow from a straightforward computation.

\begin{prop} In the conditions of Theorem \ref{main}, we have
    $$
    \mathbb E\left(\sum_{i=1}^{N}\log\sqrt{1+|z_i|^2}\right)=\frac{N}{2}.
    $$
\end{prop}

\begin{proof}
    Observe that the function $F(z)$ is in fact a \textit{polygaf} (see for example Krishnapur \cite{krishnaphd}).
    Then, following Krishnapur \cite[section 2.3]{krishnaphd}, we have that the first intensity function is
    $$
    \rho_1(z)=\frac{1}{4\pi}\Delta_z\log(K(z,z)), \qquad z\in\C,
    $$
where $\frac{1}{4}\Delta_z u=\frac{\partial^2 u}{\partial z\partial \overline{z}}$ is the complex laplacian of $u$, and $K(z,w)=(1+z\overline{w})^d$ is the covariance kernel {of each polynomial in the entries of $A$}.
Then, 
\[
\rho_1(z)=\frac{N}{\pi}\frac{1}{(1+|z|^2)^2}, \qquad z\in \C.
\]

It follows,
\[
\underset{}{\mathbb E}\left(\sum_{i=1}^{N}\log\sqrt{1+|z_i|^2}\right)=\frac{N}{\pi}\int_\C\frac{\log\sqrt{1+|z|^2}}{(1+|z|^2)^2}d\C=\frac{N}{2},
\]
where the last equality follows from taking polar coordinates.
\end{proof}

\begin{rmk}
    The previous result can be also proved by stating that the roots are invariant under the action of the unitary matrices (cf. Remark \ref{invariant_ortho}).
\end{rmk}

\begin{prop} In the conditions of Theorem \ref{main}, we have
$$
\mathbb E(\log|\det(G_d)|)=\frac{r(\psi(r+1)-1)}{2}.
$$
\end{prop}

\begin{proof}
    The result follows from recalling that $G_d$ is an $r\times r$ random matrix with i.i.d. standard Gaussian entries and applying Lemma \ref{explogdet}.
\end{proof}

Now we will focus on the second term of \eqref{tres_terminos}.

\medskip 

The random complex polynomial $F(z)$, can be seen as a stochastic process from $\C$ to $\C$.
Applying the weighted Kac-Rice formula to this particular process (see {for example 
\cite[Thm. 6.10]{AW} or \cite[Thm. 6.1]{AAL}}) one gets the following expresion,
\begin{equation}\label{eq:kacrice}
\mathbb E \left(\sum_{i=1}^{N}\log|F'(z_i)|\right)=\int_{z\in\C}\mathbb E\left(|F'(z)|^2\log|F'(z)| \bigg| F(z)=0\right)P_{F(z)}(0)dz,
\end{equation}
where $P_{F(z)}(0)$ is the density at $0$ of the random variable $F(z)$, and $\mathbb E\left(\cdot \big| F(z)=0\right)$ is the conditional expectation, conditioned on the event $\{F(z)=0\}$.

Let us define $A(z)\coloneqq \displaystyle\sum_{i=0}^d G_i{d\choose i}^{1/2}z^i$, where each $G_i$ is a $r\times r$ random matrix with i.i.d. standard Gaussian entries $\mathcal{N}_\C(0,1)$.

As a direct application of Corollary \ref{densF} we get the following.

\begin{prop}\label{densFZ}
    In the conditions of Theorem \ref{main}, the density $P_{F(z)}(0)$ satisfies
    \[
    P_{F(z)}(0)=\frac{1}{\pi (1+|z|^2)^N\Gamma(r)}.
    \]
\end{prop}

\begin{proof}
It follows from the fact that $A(z)\in\C^{r\times r}$ is a random matrix with i.i.d. centered Gaussian entries with variance $(1+|z|^2)^d$, and $F(z)=\det(A(z))$.
\end{proof}

\begin{prop}\label{condexp}
    In the conditions of Theorem \ref{main}, the conditional expectation $\mathbb E\left(|F'(z)|^2\log|F'(z)| \big| F(z)=0\right)$ is equal to
    \[
    \frac{N(1+|z|^2)^{N-2}\Gamma(r)}{2}\bigg((r+1)\psi(r+2)-r-\psi(2)+\log\left(d(1+|z|^2)^{N-2}\right)\bigg).
    \]
\end{prop}

\begin{proof}
    Taking $\varphi(A(z))=|F'(z)|^2\log|F'(z)|$, which is in fact invariant, and applying Lemma \ref{eqcondexps}, we have the following equality of conditional expectations,
    \begin{equation}\label{condexps}
    \mathbb E\left(|F'(z)|^2\log|F'(z)| \big| F(z)=0\right)=\mathbb E\left(|F'(z)|^2\log|F'(z)| \big| A_r(z)=0\right)    
    \end{equation}
    where $A_r(z)$ is the $r$-th column of $A(z)$.

    Since $F(z)=\det(A(z))$, we have that $F'(z)=\tr (\adj(A(z))A'(z))$, where $A'(z)$ is the derivative of $A(z)$.
    Using the definition of the adjugate matrix, it follows from a straightforward computation that
    \[
    \tr(\adj(A(z))A'(z))=\sum_{i=1}^r\det\left(A(z)\overset{i}{\leftarrow} A'(z)\right)
    \]
    where the matrix \(\left(A(z)\overset{i}{\leftarrow} A'(z)\right)\) is obtained from $A(z)$ replacing its $i$-th column with the $i$-th column of $A'(z)$. {If we expand this matrix in terms of its columns, we get the expression
    \[\left(A(z)\overset{i}{\leftarrow} A'(z)\right) = \left(A_1(z)\big|\cdots\big|A_{i-1}(z)\big|A'_i(z)\big| A_{i+1}(z)\big| \cdots \big| A_r(z) \right). \]
    }

    In the case where $A_r(z)=0$, we get
    \[
    F'(z)=\det\left(A(z)\overset{r}{\leftarrow} A'(z)\right)=\det\left(A_1(z)\big|\cdots\big|A_{r-1}(z)\big|A'_r(z)\right).
    \]
    In conclusion, the right-hand side of \eqref{condexps} is
    \[
    \mathbb E\left(\left|\det\left(A(z)\overset{r}{\leftarrow} A'(z)\right)\right|^2\log\left|\det\left(A(z)\overset{r}{\leftarrow} A'(z)\right)\right| \big| A_r(z)=0 \right).
    \]

    Let us do a Gaussian regression, as in Proposition 1 of Wschebor \cite{W}. 
    Write $A'_r(z)=\eta(z)+\beta A_r(z)$ such that $\eta(z)$ is a Gaussian random vector independent of $A_r(z)$.

    It is easy to check that $\mathbb E(\eta(z)\eta(z)^*)=d(1+|z|^2)^{d-2}\Id_r=\sigma_\eta ^2\Id_r$ and we have that the last conditional expectation is equal to,
    \[
    \mathbb E\left(\left|\det\left(A_1(z)\big|\cdots\big|A_{r-1}(z)\big|\eta(z)\right)\right|^2\log\left|\det\left(A_1(z)\big|\cdots\big|A_{r-1}(z)\big|\eta(z)\right)\right|\right).
    \]

    Lastly, the last expectation is equal to 
    \[
    \mathbb E\left(|\det AM|^2\log|\det AM|\right),
    \]
    where $A$ is a $r\times r$ with i.i.d. complex standard Gaussian entries, and $M$ is the diagonal matrix whose first $(r-1)$ entries are $(1+|z|^2)$ the last one is $\sigma_\eta$.

    The result follows from applying Corollary \ref{corexpdet}.
\end{proof}

\medskip

Now that we have computed successfully the density $P_{F(z)}(0)$ and the conditional expectation $\mathbb E\left(|F'(z)|^2\log|F'(z)| \bigg| F(z)=0\right)$, let us compute the remaining term of \eqref{tres_terminos}.

\begin{prop} In the conditions of Theorem \ref{main},
$$ 
\mathbb E \left(\sum_{i=1}^{N}\log|F'(z_i)|\right)=\frac{N}{2}\bigg(N+\log d + (r+1)\psi(r+2)-r-\psi(2)-2\bigg).
$$
\end{prop}

\begin{proof}
    Applying the Kac-Rice formula \eqref{eq:kacrice}, we get that the left-hand side is equal to
    $$
    \int_{z\in\C}\mathbb E\left(|F'(z)|^2\log|F'(z)| \bigg| F(z)=0\right)P_{F(z)}(0)dz.
    $$
    Furthermore, applying Proposition \ref{condexp} and Proposition \ref{densFZ} we have that it is equal to
    $$
    \frac{N}{2\pi}\left(\int_{z\in\C}\frac{(r+1)\psi(r+2)-r-\psi(2)+\log d}{(1+|z|^2)^2}dz+(N-2)\int_{z\in\C}\frac{\log (1+|z|^2)}{(1+|z|^2)^2}dz\right).
    $$
    Then, the result follows from changing to polar coordinates.
\end{proof}

\subsection{Proof of Main Theorem}\label{proofmain}

Now, we will combine everything done up to now, to prove Theorem \ref{main}.

\medskip
Let $F(z)$ be the complex polynomial of degree $N=rd$, defined as
\[
F(z)=\det\left(\sum_{i=0}^d G_i{d\choose i}^{1/2}z^i\right),
\]
where $G_i$ are $r \times r$ complex random matrices with i.i.d. standard Gaussian entries $\mathcal{N}_\C(0,1)$.

We want to compute $\mathbb E(V(\hat z_1,\cdots,\hat z_N))$, the expected value  of the logarithmic energy of the projection of the roots of $F(z)$ onto the Riemmann sphere. 

\medskip 

Recall \eqref{tres_terminos}, we have that $\mathbb E(V(\hat z_1,\cdots,\hat z_N))$ is equal to
\[
(N-1)\mathbb E\left(\sum_{i=1}^{N}\log\sqrt{1+|z_i|^2}\right)-\frac{1}{2}\mathbb E\left(\sum_{i=1}^{N}\log|F'(z_i)|\right)+\frac{N}{2}\mathbb E(\log|\det(G_d)|),
\]
which, thanks to the computations of Section \ref{computations}, is in turn equal to
\[
\frac{(N-1)N}{2}-\frac{N}{4}\bigg(N+\log d + (r+1)\psi(r+2)-r-\psi(2)-2\bigg)+\frac{Nr(\psi(r+1)-1)}{4}.
\]

In conclusion,
\[
\mathbb E(V(\hat z_1,\cdots,\hat z_N))=\frac{N^2}{4}-\frac{N\log d}{4}-\frac{N}{4}\bigg(1+\psi(r+1)-\psi(2)\bigg).
\]

The result in $\S^2$ follows by substracting $\frac{N(N-1)}{2}\ln 2$, see \eqref{exps2}.

\qed

\section{Discussion and experimental examples}\label{sec:discussion}

As observed above, the PEVP ensemble contains the spherical ensemble \cite{alishahi2015spherical} and the random polynomial roots \cite{ABS} as particular examples. Remember that the constant for the linear term in the energy was $\left(\frac{\ln 2}{2} -\frac{\gamma}{4} \right) \approx 0.2022\ldots$ for the spherical ensemble, and $-\frac{\kappa}{2} \approx 0.096\ldots$ for the random polynomial roots. This means that the random polynomial roots are better distributed than the spherical ensemble in terms of the logarithmic energy. 

In this section, we study the dependence on $d$ of the logarithmic energy of the PEVP ensemble in order to see if for an intermediate $(d,r)$, the logarithmic energy lies between the two extremes described above. From Theorem \ref{main}, we have that the expectation of the logarithmic energy is
$$
\E(V(x_1,\cdots,x_N))=\frac{\kappa}{2}N^2-\frac{N\log d}{4}-\frac{N}{4}\bigg(1+\psi(r+1)-\psi(2) -2\ln 2\bigg).
$$

Now, it is well known that
\begin{equation}\label{psi_sum}
    \psi(r+1) = \frac{\Gamma'(r+1)}{\Gamma(r+1)} = -\gamma+ \sum_{j=1}^{r}\frac{1}{j},
\end{equation}
and using the Euler-Maclaurin formula as in \cite{alishahi2015spherical} we have
\begin{equation}\label{euler_maclaurin}
    \sum_{j=1}^{r}\frac{1}{j} = \ln r + \gamma + \frac{1}{2r} - \frac{1}{12r^2} + O\left(\frac{1}{r^4} \right). 
\end{equation}
Combining \eqref{psi_sum} and \eqref{euler_maclaurin} we have
\begin{equation}\label{psi_expression}
    \psi(r+1) = \ln r + \frac{1}{2r} - \frac{1}{12r^2} + O\left(\frac{1}{r^4} \right). 
\end{equation}
Using \eqref{psi_expression} and the equality $\gamma = 1-\psi(2)$ in the logarithmic energy expression, we obtain
\begin{eqnarray*}
\E(V(x_1,\cdots,x_N))&=
& \frac{\kappa}{2}N^2-\frac{N\ln N}{4}+N\left(\frac{\ln 2}{2}-\frac{\gamma}{4}\right) -\frac{d}{8} + \frac{d^2}{48N} + N O\left(\frac{1}{r^4} \right).
\end{eqnarray*}

Let us make two observations. Firstly, for $d=1$ we recover the expression for the logarithmic energy of the spherical ensemble. Secondly, for a given $N$, the logarithmic energy of the PEVP ensemble decreases linearly with $d$. Indeed, the term $\frac{d^2}{48N}$ is comparable with $\frac{d}{8}$ only for $d \geq 6N$, which is obviously not the case if $d$ is a factor of $N$.

In what follows, we present some experimental results which illustrate this dependence with $d$, and the comparison of the computed expected value of the logarithmic energy with the empirical values.

We take two values for the number of points, $N= 60$ and $N=120$, which have several divisors, namely $12$ and $16$, respectively. For each $N$ and for each pair $(d,r)$ such that $N=dr$, we sort random matrices and solve numerically the corresponding PEVP. The obtained eigenvalues are projected to $\S^2$, and we compute the resulting logarithmic energy. This process is repeated $100.000$ times for each pair $(d,n)$. 

In Figure \ref{fig:todojunto}, we present a violin plot (a kernel density estimation) for each $d$, on top of the expected value computed in Theorem \ref{main}. Notice that for $d=N$, the logarithmic energy corresponds to the roots of random polynomials \cite{ABS}.

In Figure \ref{fig:dependencia_en_d}, we show, with the same experimental data, the linear dependence of the logarithmic energy on $d$.

In Figure \ref{fig:diferencias_promedios}, we show the difference between the expected value computed in Theorem \ref{main} and the density estimated in the $100.000$ problem samples for each $d$. Observe that the distributions are centered at zero, and the variance seems to decrease with $d$.

\begin{figure}[h!]
    \centering
    \includegraphics[width=0.49\linewidth]{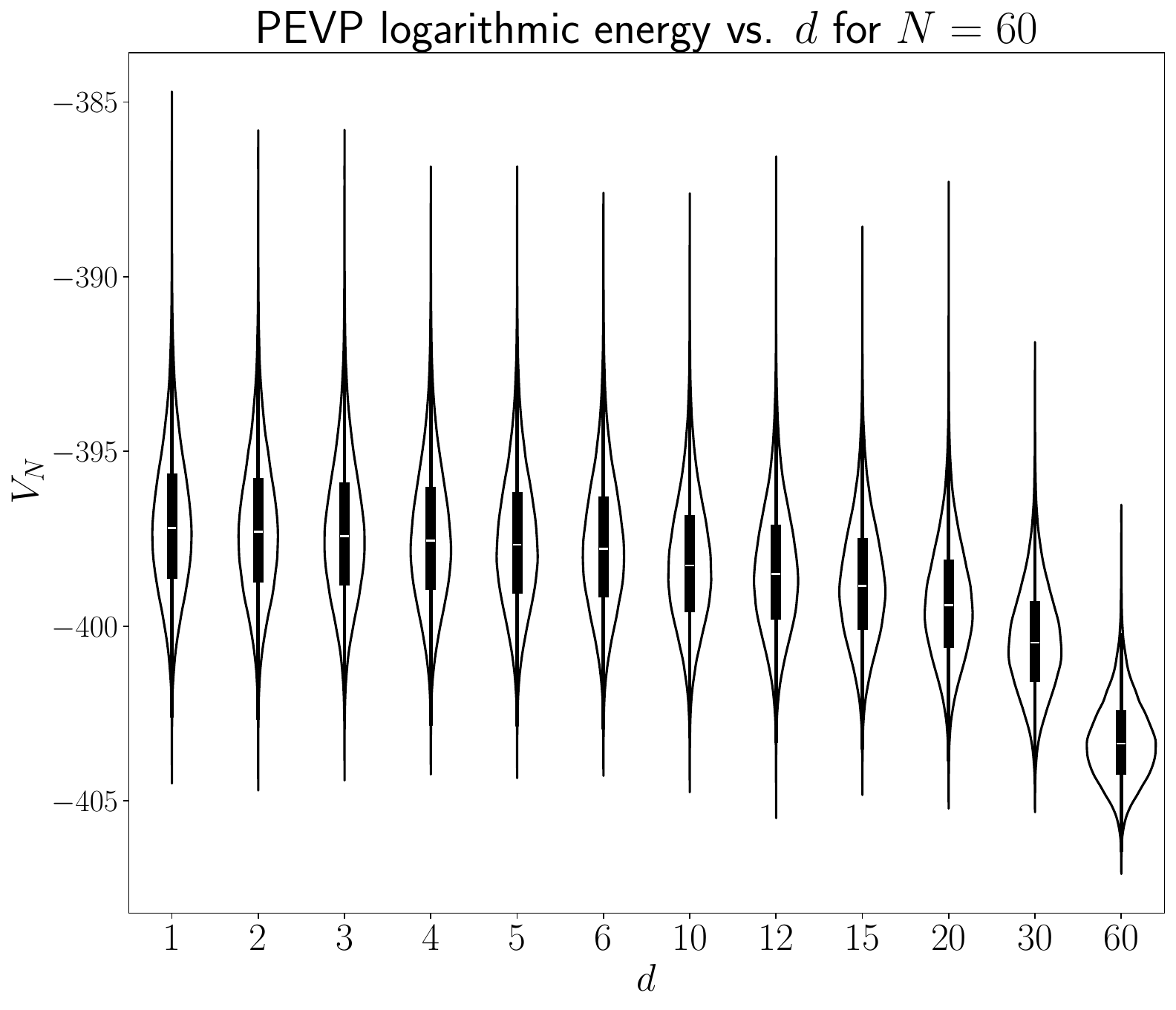}
    \includegraphics[width=0.49\linewidth]{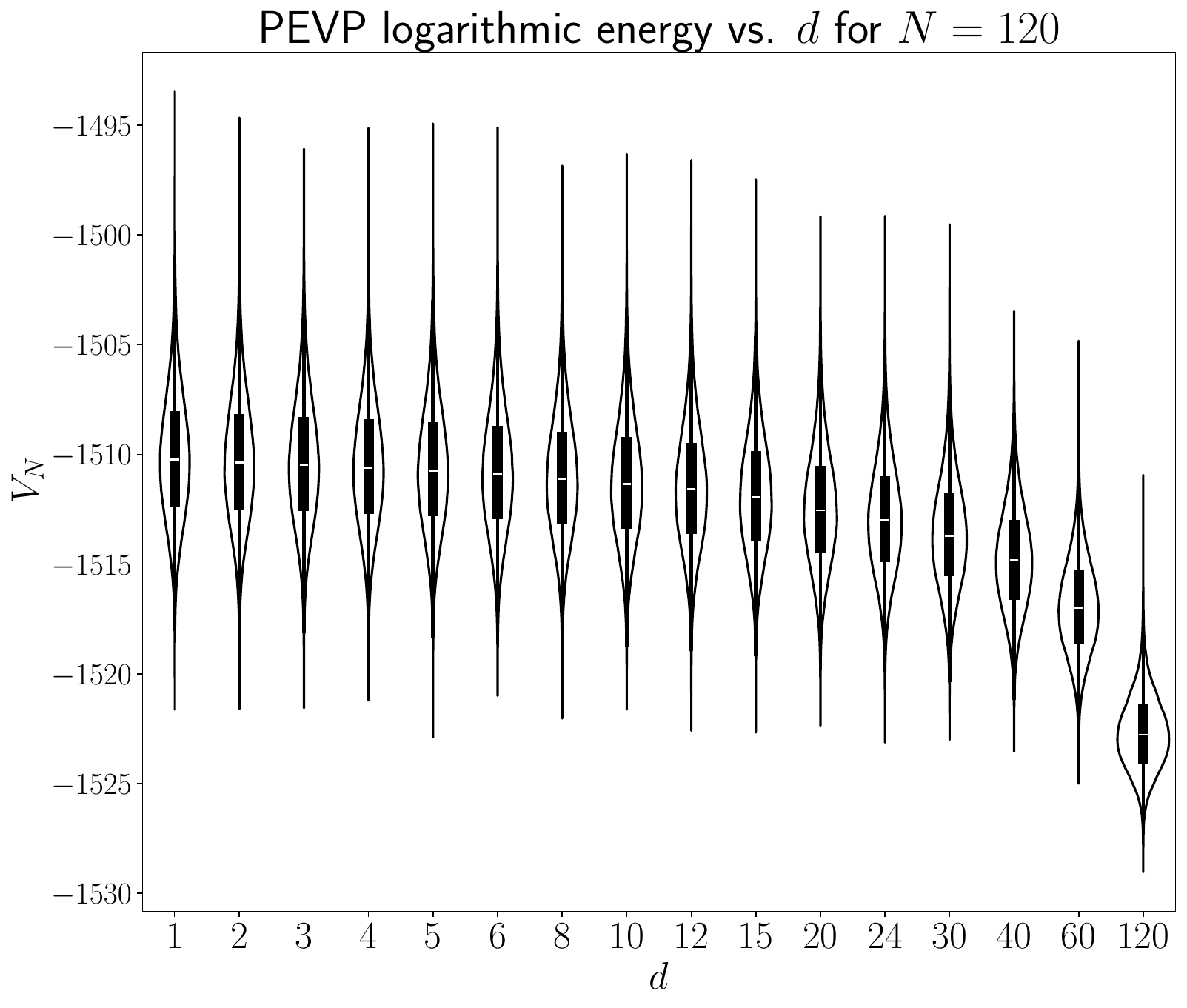}
    \caption{Empirical logarithmic energies for PEVP ensambles. The violin plots are computed using $100.000$ repetitions for each pair $(r,d)$.
    }
    \label{fig:todojunto}
\end{figure}
\begin{figure}[h!]
    \centering
    \includegraphics[width=1\linewidth]{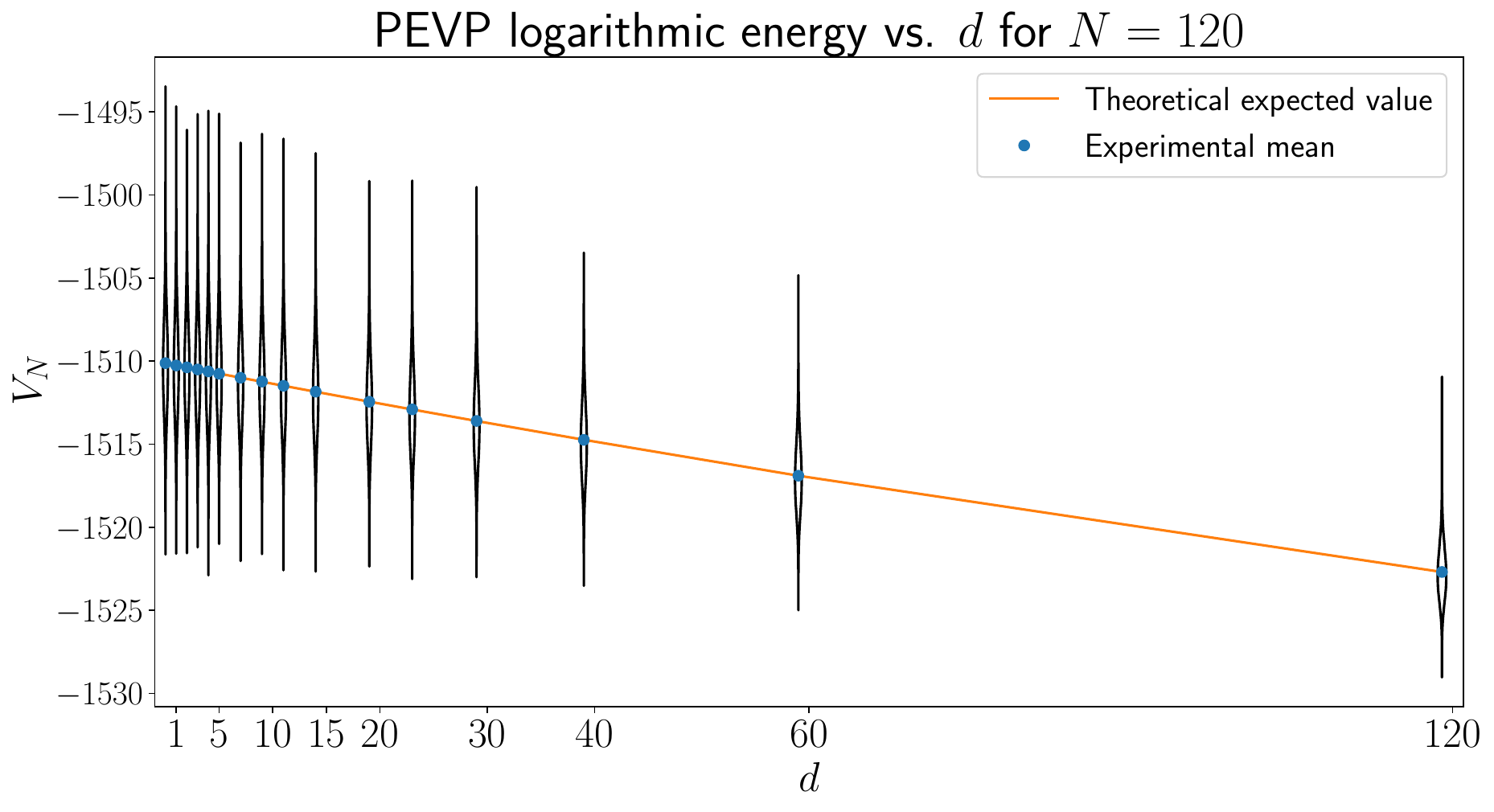}
    \caption{Dependence of the logarithmic energy on $d$. The empirical results are the same as in Fig. \ref{fig:todojunto}, but with $d$ now correctly scaled on the x-axis, and the expected value from Theorem \ref{main} included.}
    \label{fig:dependencia_en_d}
\end{figure}
\begin{figure}[h!]
    \centering
    \includegraphics[width=1\linewidth]{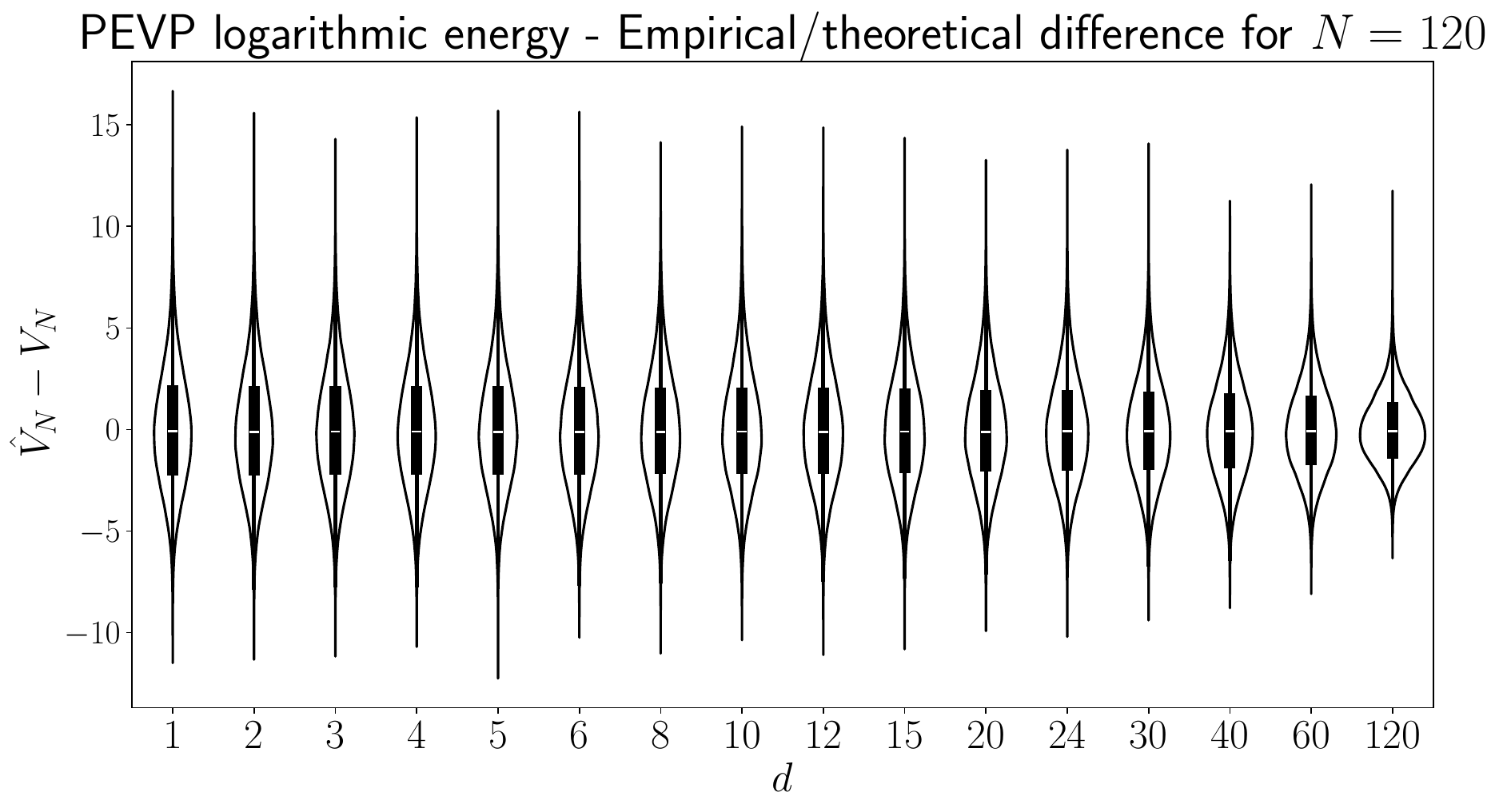}
    \caption{Difference between the empirical results and the expected value of Theorem \ref{main}. Notice how the difference is centered at zero for all $d$.}
    \label{fig:diferencias_promedios}
\end{figure}

\newpage
\bibliographystyle{plain}
\bibliography{biblio}

\begin{thebibliography}{10}

\bibitem{alishahi2015spherical}
Kasra Alishahi and Mohammadsadegh Zamani.
\newblock The spherical ensemble and uniform distribution of points on the
  sphere.
\newblock {\em Electron. J. Probab.}, 20:no. 23, 27, 2015.

\bibitem{AAL}
Diego Armentano, Jean-Marc Azaïs, and José~Rafael León.
\newblock On a general {Kac-Rice} formula for the measure of a level set.
\newblock {\em Annals of Applied Probability (Accepted),
  https://arxiv.org/abs/2304.07424}, 2025.

\bibitem{ABS}
Diego Armentano, Carlos Beltr\'{a}n, and Michael Shub.
\newblock Minimizing the discrete logarithmic energy on the sphere: the role of
  random polynomials.
\newblock {\em Trans. Amer. Math. Soc.}, 363(6):2955--2965, 2011.

\bibitem{arnold}
Vladimir~Igorevich Arnold, Sabir~Medgidovich Gusein-Zade, and Alexander
  Varchenko.
\newblock {\em Singularities of differentiable maps. {V}olume 1}.
\newblock Modern Birkh\"auser Classics. Birkh\"auser/Springer, New York, 2012.
\newblock Classification of critical points, caustics and wave fronts,
  Translated from the Russian by Ian Porteous based on a previous translation
  by Mark Reynolds, Reprint of the 1985 edition.

\bibitem{AW}
Jean-Marc Aza\"{\i}s and Mario Wschebor.
\newblock On the roots of a random system of equations. {T}he theorem on {S}hub
  and {S}male and some extensions.
\newblock {\em Found. Comput. Math.}, 5(2):125--144, 2005.

\bibitem{beltran2018diamond}
Carlos Beltr\'{a}n and Uju\'{e} Etayo.
\newblock The diamond ensemble: a constructive set of spherical points with
  small logarithmic energy.
\newblock {\em J. Complexity}, 59:101471, 22, 2020.

\bibitem{carlos_fatima}
Carlos Beltr\'an and F\'atima Lizarte.
\newblock A lower bound for the logarithmic energy on {$\Bbb{S}^2$} and for the
  green energy on {$\Bbb{S}^n$}.
\newblock {\em Constr. Approx.}, 58(3):565--587, 2023.

\bibitem{betermin2018renormalized}
Laurent B\'{e}termin and Etienne Sandier.
\newblock Renormalized energy and asymptotic expansion of optimal logarithmic
  energy on the sphere.
\newblock {\em Constr. Approx.}, 47(1):39--74, 2018.

\bibitem{BCSS}
Lenore Blum, Felipe Cucker, Michael Shub, and Steve Smale.
\newblock {\em Complexity and real computation}.
\newblock Springer-Verlag, New York, 1998.
\newblock With a foreword by Richard M. Karp.

\bibitem{brauchart2012}
Johann Brauchart, {Douglas} Hardin, and {Edward} Saff.
\newblock The next-order term for optimal {R}iesz and logarithmic energy
  asymptotics on the sphere.
\newblock In {\em Recent advances in orthogonal polynomials, special functions,
  and their applications}, volume 578 of {\em Contemp. Math.}, pages 31--61.
  Amer. Math. Soc., Providence, RI, 2012.

\bibitem{Gradshteyn}
Izrail~Solomonovich Gradshteyn and Iosif~Moiseevich Ryzhik.
\newblock {\em Table of integrals, series, and products}.
\newblock Elsevier/Academic Press, Amsterdam, eighth edition, 2015.
\newblock Translated from the Russian, Translation edited and with a preface by
  Daniel Zwillinger and Victor Moll.

\bibitem{krishna}
J.~Ben Hough, Manjunath Krishnapur, Yuval Peres, and B\'{a}lint Vir\'{a}g.
\newblock {\em Zeros of {G}aussian analytic functions and determinantal point
  processes}, volume~51 of {\em University Lecture Series}.
\newblock American Mathematical Society, Providence, RI, 2009.

\bibitem{krishnaphd}
Manjunath~Ramanatha Krishnapur.
\newblock {\em Zeros of random analytic functions}.
\newblock ProQuest LLC, Ann Arbor, MI, 2006.
\newblock Thesis (Ph.D.)--University of California, Berkeley.

\bibitem{Lauritsen}
Asbj\o rn B\ae~kgaard Lauritsen.
\newblock Floating {W}igner crystal and periodic jellium configurations.
\newblock {\em J. Math. Phys.}, 62(8):Paper No. 083305, 17, 2021.

\bibitem{Fluctuations}
Marcus Michelen and Oren Yakir.
\newblock Fluctuations in the logarithmic energy for zeros of random
  polynomials on the sphere.
\newblock {\em Probability Theory and Related Fields}, pages 1--58, 10 2024.

\bibitem{B3}
Michael Shub and Steve Smale.
\newblock Complexity of {B}ezout's theorem. {III}. {C}ondition number and
  packing.
\newblock volume~9, pages 4--14. 1993.
\newblock Festschrift for Joseph F. Traub, Part I.

\bibitem{smale_problems}
Steve Smale.
\newblock Mathematical problems for the next century.
\newblock In {\em Mathematics: frontiers and perspectives}, pages 271--294.
  Amer. Math. Soc., Providence, RI, 2000.

\bibitem{W}
Mario Wschebor.
\newblock On the {K}ostlan-{S}hub-{S}male model for random polynomial systems.
  {V}ariance of the number of roots.
\newblock {\em J. Complexity}, 21(6):773--789, 2005.

\end{thebibliography}



\end{document}